\documentclass[12pt]{article}
\usepackage[cp1250]{inputenc}
\usepackage{fancyhdr}

\usepackage{amsmath}
\usepackage{amsfonts}
\usepackage{amsthm}
\usepackage{amssymb}
\usepackage{color}

\newcommand{\barint}{
         \rule[.036in]{.12in}{.009in}\kern-.16in
          \displaystyle\int  }
\newcommand{\R}{{\mathbb{R}}}

\newcommand{\rn}{{\mathbb{R}^{n}}}
\newcommand{\sph}{{\rm supp\,\phi}}

\newcommand{\Om}{{\Omega}}

\newcommand{\Oc}{{\Omega\cap}}
\newcommand{\al}{{\alpha}}
\newcommand{\be}{{\beta}}

\newcommand{\D}{{\Delta}}

\newcommand{\s}{{\sigma}}

\newcommand{\ve}{\varepsilon}

\newtheorem{theo}{\bf Theorem}[section]
\newtheorem{coro}{\bf Corollary}[section]
\newtheorem{lem}{\bf Lemma}[section]
\newtheorem{rem}{\bf Remark}[section]
\newtheorem{defi}{\bf Definition}[section]

\newtheorem{remark}{\bf Remark}[section]

\newcommand{\essinf}{\mathop{\textrm{ess\,inf}}}
\newcommand{\esssup}{\mathop{\textrm{ess\,sup}}}

\title{Hardy--type inequality  in variable exponent Lebesgue spaces derived from nonlinear problem}

\usepackage{authblk}

\author[$\dagger$]{Sylwia Dudek}
\author[$\diamond$]{Iwona Skrzypczak\footnote{The author was supported by NCN grant 2011/03/N/ST1/00111.}}

\affil[$\dagger$]{\small
Institute of Mathematics,
Krakow University of Technology, \newline
ul. Warszawska 24, 31-155 Krakow, Poland \newline e--mail: sbarnas@pk.edu.pl}

\affil[$\diamond$]{\small
Faculty of Mathematics, Informatics and Mechanics,
University of Warsaw, \newline
ul. Banacha 2, 02-097 Warsaw, Poland \newline e--mail: iskrzypczak@mimuw.edu.pl}
\date{}

\begin{document}
\maketitle
\thispagestyle{empty}

\begin{abstract}
We derive a family of weighted Hardy--type inequalities in the variable exponent Lebesgue space with an additional term of the form
 \[
\int_\Om \ |\xi|^{p(x)} \mu_{1,\beta}(dx)\leqslant \int_\Om |\nabla \xi|^{p(x)}\mu_{2,\beta}(dx)+\int_\Om \left|\xi{\log \xi} \right|^{p(x)}  \mu_{3,\beta}(dx),
\]
where $\xi$ is any compactly supported Lipschitz function. The involved measures depend on a certain solution to
 the partial differential inequality involving $p(x)$--Laplacian ${-}\Delta_{p(x)} u\geqslant \Phi$, 
 defined on an open and not necessarily bounded  subset $\Omega\subseteq\rn $, and a~certain parameter~$\be$. We derive new Caccioppoli--type inequality for the solution $u$.  As its consequence we get Hardy--type inequality.

We present the derivation of the family of weighted Hardy--type inequalities in $\Omega\subseteq\rn $. We illustrate the result by several one--dimensional examples.
The paper extends the recent results of the second author which imply classical Hardy and Hardy--Poincar\'{e} inequalities with the optimal constants.
\end{abstract}


\medskip


\section{Introduction}\label{intro}
In this paper  we derive a family of Hardy--type inequalities with variable exponent of the form \begin{equation}\label{H-P}
\int_\Om \ |\xi|^{p(x)} \mu_{1,\beta}(dx)\leqslant \int_\Om |\nabla \xi|^{p(x)}\mu_{2,\beta}(dx)+\int_\Om \left|\xi{\log \xi} \right|^{p(x)} \frac{\left|\nabla p(x)\right|^{p(x)}}{{p(x)}^{p(x)}} \mu_{2,\beta}(dx),
\end{equation}
where $\Om$ is an open subset of $\rn$,  not necessarily bounded, the exponent $p$ is such that $p \in W_{loc}^{1,1}(\Omega)$, ${p}^{p(x)},|\nabla p|^{p(x)}\in L^{1}_{loc}(\Omega)$ and satisfies $1<\textrm{ess}\inf_{x \in \Omega} p(x) \leqslant p(x) \leqslant \textrm{ess}\sup_{x \in \Omega} p(x)<\infty$, and a~function $\xi:\Om\to\R$ is  compactly supported and  Lipschitz. The involved measures $\mu_{1,\beta}(dx), \mu_{2,\beta}(dx)$ depend on $p(x)$, a~certain parameter $\be$, a~continuous function $\s(x)$, and a~nonnegative weak solution $u$ to the PDI
\begin{equation}\label{plapxproblem}
-\Delta_{p(x)}u\geqslant \Phi \quad \mathrm{in}\quad \Om,
\end{equation}
with a locally integrable function $\Phi$. We admit the functions $\s(x)$ and $\Phi$ satisfying compatibility conditions with $p(x)$ (see crucial conditions).


We deal with the variable exponent Lebesgue spaces, which recently have received more and more attention both --- from the theoretical and from the applied point of view. We refer to~\cite{ks,KoRa} for the detailed information on the theoretical approach to the Lebesgue and the Sobolev spaces with variable exponents. Various attempts to prove existence, uniqueness or regularity theory for problems stated in variable exponent spaces can be found e.g.~in~\cite{barnas, fanzhang}. We refer for the survey \cite{overview}  summarising inter alia results on qualitative properties of solutions to the related PDEs. We mark that the variable exponent Lebesgue spaces are investigated since 1930s when Orlicz introduced them in~\cite{Orlicz}. They are under permanent development by various groups of mathematicians~\cite{cruz,Hudzik,nakano1,nakano2}.

 The typical examples of equations stated in variable exponent spaces are models of electrorheological fluids, see e.g.~\cite{raj-ru1,raj-ru2,el-rh2}. This kind of materials have been intensively investigated recently. Electrorheological fluids change their mechanical properties dramatically when an external electric field is applied, so the variable exponent Lebesgue setting is natural for their modelling.
Some classical models are also generalised in the variable exponent Lebesgue spaces. In \cite{ks} we find investigations on Poisson equation, as well as Stokes problem being of fundamental importance in describing fluid dynamics. Let us mention that various models require different types of restrictions on $p(x)$, therefore the unified approach is missing.

Hardy--type inequalities are important tools in various fields of analysis.  Let us mention such branches as functional analysis, harmonic analysis, probability theory, and PDEs. Hardy--type inequalities are investigated on their own in the classical way~\cite{kmp,muckenhoupt,plap}, as well as in the various generalised frameworks~\cite{bogdan,bhs,buc,akkpp2012,orliczhardy}.

Recently, Hardy--type inequalities in the variable exponent Lebesgue spaces have become a lively studied topic of analysis~\cite{AdamHa,cruzuribe,DiSa,HaHaKo,Ha,harman2,HaMa,harman3,MaHa,MaZe1,MaZe2,MCO,miz1,MCMO,RaSa,Samko1}. The most common idea in these papers is investigating links between validity of Hardy--type inequalities and boundedness of maximal operator.  One--dimensional case is considered in~\cite{DiSa,MCO,MCMO}, where the exponents are possibly different on the right-- and the left--hand side of the inequality. The paper~\cite{Samko1} is devoted to the inequality with the weights depending on distance from a single point, while in~\cite{HaHaKo,miz1} the weights depend on distance from a boundary in $\R^n$.  There are several papers~\cite{cruzuribe,Ha,harman2,HaMa,harman3,MaZe1,MaZe2} dealing with the necessary and sufficient conditions for the validity of Hardy inequality involving  Hardy operator. Different approach we find in~\cite{Bo}, where the authors investigate the class of admissible weights for Hardy--type inequality holding for  nonincreasing functions.

We point out that in the majority of the above papers the authors deal with the norm version of Hardy--type inequality. We obtain the modular one, which is stronger. We would like to stress that only in the constant exponent case the both types are equivalent. In the variable exponent case it is not direct to transform one of these types to another. To the authors' best knowledge the only result of this kind is given by Fan--Zhao~\cite[Theorem~1.3]{orlicz} where the authors derive a tool giving certain form of the norm version of Hardy inequality from a modular one.

The purpose of this paper is to introduce a new tool for derivation of Hardy--type inequalities with variable exponent on the basis of nonlinear problems. The idea of similar constructions in the constant exponent is present in a few papers. In~\cite{barbhardy04} Barbatis, Filippas, and Tertikas derive Hardy--type inequalities on a domain where certain power of the function expressing distance from the boundary is $p$--superharmonic. In~\cite{dambr} D'Ambrosio obtaines an inequality related to~\eqref{H-P} as a consequence of the inequality $-\Delta_p(u^\alpha) \geqslant 0$ with a certain constant~$\alpha$. Similar approach can be found in~\cite{plap} by the second author.

Our considerations are based  on  the methods introduced in~\cite{nonex,pohmi_99} and developed in~\cite{plap,bcp-plap,orliczhardy} in various ways.
 In~\cite{nonex} the authors investigate nonexistence of nontrivial nonnegative  weak solutions to $A$--harmonic problems starting with derivation of Caccioppoli--type estimate for their weak solutions. As a starting point to derive Hardy--type inequality we focus on this step.
 We modify the proof of Theorem~4.1 from~\cite{plap}, where the investigated PDI reads
\begin{equation}\label{problem}
-\Delta_pu\geqslant \Phi \quad \mathrm{in}\quad \Om,
\end{equation}
with a locally integrable function $\Phi$ being in a certain sense not very negative. This condition generalises the requirement that the solution $u$ is supposed to be a $p$--superharmonic function. As it is shown in~\cite{plap}, the substitution in the derived Caccioppoli--type inequality for solutions implies the family of Hardy--type inequalities of the form
\begin{equation*}
\int_\Omega \ |\xi|^p \mu_{1,\beta}(dx)\leqslant \int_\Omega |\nabla \xi|^p\mu_{2,\beta}(dx),
\end{equation*}
where $1<p<\infty$, $\xi:\Omega\to\R$ is  compactly supported Lipschitz function, and $\Omega$ is an open subset of $\rn$. The involved measures $\mu_{1,\beta}(dx)$, $\mu_{2,\be}(dx)$ depend on a certain parameter $\beta$ and on $u$ --- a nonnegative weak solution to~\eqref{problem}. Among other results it implies classical Hardy and Hardy--Poincar\'{e} inequalities with optimal constants (see~\cite{plap,bcp-plap}, respectively). We retrieve the main result of~\cite{plap}
as a special case here (see Theorem~\ref{theoplap}) and therefore we confirm all the examples  from~\cite{plap,bcp-plap}.

We extend the techniques from~\cite{plap} to the more general case when we deal with~\eqref{plapxproblem} instead of~\eqref{problem}. We derive Hardy--type inequality  in the variable exponent Lebesgue spaces on $\R^n$. Then we pay particular attention to the case of $n=1$, because it is easier to  compare with many existing one--dimensional results, e.g.~\cite{Bo,DiSa,HaHaKo,MCO,MCMO}. Moreover, higher dimensional problems may be reduced to this case, when we assume certain kind of symmetry.  The paper~\cite{barskrzy2} is devoted to further analysis of the results of our paper in $\R^n$. We hope that our result will be found useful in applied mathematics, especially in investigations on qualitative properties of solutions to nonlinear problems.

The paper is organised as follows. Section~\ref{CacSec} is devoted to derivation of Caccioppoli--type inequality for solutions to~\eqref{plapxproblem}. In Section~\ref{SecHar} we derive general $p(x)$--Hardy inequality for compactly supported Lipschitz functions. In Section~\ref{Sec1d} we concentrate on inequalities in one dimension. In Section~\ref{SecLink} we give detailed comparison with the results  existing in the literature. We conclude our paper in Section~\ref{SecOpen} by posing open questions.

\section{Preliminaries}\label{prelim}
\subsubsection*{Notation}

In the sequel we assume that $\Omega\subseteq\mathbb{R}^n$ is an open subset not necessarily bounded. If $f$ is defined on the set $A$ by $f\chi_{A}$ we understand function $f$ extended by $0$ outside $A$.
By $\langle\cdot,\cdot\rangle$ we understand the classical scalar product in~$\R^n$. We say that the function $f$ has values separated from $0$, if there exists a constant $c_0$ such that $f(x)\geqslant c_0>0$ for every $x$.


\subsubsection*{General Lebesgue and Sobolev spaces}

In the sequel we suppose that measurable function $p:\Omega \rightarrow (1, \infty)$ is such that
	\begin{equation}\label{P}
  1<p^-:=\essinf_{x \in \Omega} p(x) \leqslant p(x) \leqslant p^+:=\esssup_{x \in \Omega} p(x)<\infty.
	\end{equation}

We recall some properties of the variable exponent
 spaces $L^{p(x)}(\Omega)$ and $W^{1,p(x)}(\Omega)$. By
$E(\Omega)$ we denote the set of all equivalence classes of measurable real functions defined on $\Omega$ being equal almost everywhere. The variable exponent Lebesgue space is defined as
\[
    L^{p(x)}(\Omega)=\{u \in E(\Omega): \int_{\Omega} |u(x)|^{p(x)} dx<\infty \}
\]
  equipped with the Luxemburg--type norm
\[
    \|u\|_{L^{p(x)}(\Omega)}:=\inf \Big\{\lambda>0: \int_{\Omega} \Big|\frac{u(x)}{\lambda}\Big|^{p(x)}dx \leqslant 1 \Big\}.\]
We define the variable exponent Sobolev space by
\[
    W^{1,p(x)}(\Omega) = \{u \in L^{p(x)}(\Omega): \nabla u \in L^{p(x)}(\Omega; \mathbb{R}^n)\}
\]
  equipped with the norm $\|u\|_{W^{1,p(x)}(\Omega)} = \|u\|_{L^{p(x)}(\Omega)}+\|\nabla u\|_{L^{p(x)}(\Omega)}.$

  Then $(L^{p(x)}(\Omega),\|\cdot\|_{L^{p(x)}(\Omega)})$ and $(W^{1,p(x)}(\Omega), \|\cdot\|_{W^{1,p(x)}(\Omega)})$ are
  separable and reflexive Banach spaces.

  For more detailed information we refer to 
	\cite{ks,fan,orlicz}.
	
\medskip

By $\mathcal{P}(\Omega)$	we denote the class of the functions $p$ such that~\eqref{P} is satisfied and $p \in W_{loc}^{1,1}(\Omega)$, ${p}^{p(x)},|\nabla p|^{p(x)}\in L^{1}_{loc}(\Omega)$.
	
\subsubsection*{Differential inequality}

Our analysis is based on the following differential inequality.

\begin{defi}\label{defnier}
Let $\Om$ be any open subset of $\rn$. We assume that the measurable function $p:\Omega \rightarrow (1, \infty)$ satisfies~\eqref{P}
and $\Phi$ is the locally integrable  function defined in $\Om$ such that  for every nonnegative compactly supported $w\in
W^{1,p(x)}(\Om)$, we have
$
 \int_\Om \Phi w\,dx >-\infty.
$\\
  Let $u\in W^{1,p(x)}_{loc}(\Om)$ and $u\not\equiv 0$. We say that
\[
-\Delta_{p(x)} u\geqslant \Phi,
\]
if for every nonnegative compactly supported $w\in
W^{1,p(x)}(\Om)$, we have
\begin{equation}\label{nikfo}
\langle -\Delta_{p(x)} u,w \rangle := \int_\Om |\nabla u|^{p(x)-2}\langle\nabla
u,\nabla w\rangle\, dx \geqslant \int_\Om \Phi w\, dx.
\end{equation}
\end{defi}
\begin{rem}\label{mal1}\rm
Note that $p(x)$--Laplacian is a continuous, bounded, and strictly monotone operator defined for every compactly supported function $w \in W^{1,p(x)}(\Om)$ (see e.g. \cite[Theorem 3.1]{fanzhang} for the definitions and the proofs).
In particular, it is well--defined in the distributional sense.
\end{rem}

\subsubsection*{Crucial conditions} 

We suppose that the measurable function $p:\Omega \rightarrow (1,\infty)$ satisfies~\eqref{P}, nonnegative $u
\in W^{1,p(x)}_{loc}(\Om)$ and $\Phi\in L^{1}_{loc}(\Om)$ satisfy PDI $-\Delta_{p(x)} u\geqslant \Phi$, in the sense of Definition~\ref{defnier}. 
 We assume that there exist a continuous function $\s(x):\overline{\Omega}\to\R$ and a parameter $\beta>0$, such that the following conditions are satisfied
\begin{eqnarray}
&\label{sx}{\Phi\cdot u}+{\s(x) |\nabla u|^{p(x)}}\geqslant 0 \quad{\rm a.e.\  in}\quad \Om,&\\
&\be>\sup\limits_{x \in \overline{\Omega}} \s(x).&\label{sbeta}
\end{eqnarray}

\section{Caccioppoli estimate for solution of differential inequality $-\D_{p(x)} u\geqslant \Phi$}\label{CacSec}

Before we formulate the main theorem of this section we state the following useful lemmas.
	
\begin{lem}\label{nonex1}
 Let $u
\in W^{1,p(x)}_{loc}(\Om)$, $u>0$ and $\phi$ be a nonnegative Lipschitz function with
compact support in $\Om$ such that the integral $\int_{\sph}{|\nabla\phi|}^{p(x)}\phi^{1-p(x)}\, dx$ is finite.
We fix
 $0<\delta<R$, $\be>0$ and denote
\begin{eqnarray} \label{G}
u_{\delta,R}(x):= {\rm min}\left\{ u(x)+\delta, R \right\},
 & & G(x):= (u_{\delta,R}(x))^{-\be}\phi (x).
\end{eqnarray}
Then $u_{\delta,R}\in
W^{1,p(x)}_{loc}(\rn )$ and $G\in
W^{1,p(x)}(\Om)$.
\end{lem}

\begin{remark}\rm
See e.g. \cite[Proposition 8.1.9]{ks}, to obtain $u_{\delta,R}\in
W^{1,p(x)}_{loc}(\rn )$. We note that the truncated function satisfies $\delta \leqslant u_{\delta,R}(x) \leqslant R$ and therefore we have $(u_{\delta,R}(x))^{-\be} \in W^{1,p(x)}_{loc}(\rn )$.  The function $G$ is compactly supported,  thus $G\in
W^{1,p(x)}(\Om)$.
\end{remark}

\begin{lem}\label{lemmjednor}
Let a measurable function $p:\Omega \rightarrow (1, \infty)$ satisfy~\eqref{P}, a function $\tau(x):\Omega \rightarrow \R_+$ be continuous, bounded, with values separated from $0$, and $s_1,s_2\geqslant 0$. Then for a.e. $x\in\Om$ we have
\[s_1s_2^{p(x)-1}\leqslant \frac{1}{p(x) \tau(x)^{p(x)-1}}\cdot  s_1^
{p(x)} +\frac{p(x)-1}{p(x)} \tau(x) \cdot s_2^{p(x)} .\]
\end{lem}
\begin{proof}[\textbf{Proof}]We apply  classical Young inequality
$ab\leqslant \frac{a^{p(x)}}{p(x)}+\frac{p(x)-1}{p(x)}b^\frac{p(x)}{p(x)-1}$ with $a=\frac{s_1}{\eta(x)^{p(x)-1}},\ b=(s_2\eta(x))^{p(x)-1}$, where $\eta(x)$ is an arbitrary continuous, bounded function with values separated from $0$, to get

\begin{eqnarray*}
s_1s_2^{p(x)-1}&=&\left(\frac{s_1}{\eta^{p(x)-1}}\right)(s_2\eta)^{p(x)-1}\leqslant\\
&\leqslant&\frac{1}{p(x)}\left(\frac{s_1}{\eta^{p(x)-1}}\right)^{p(x)}+
\frac{p(x)-1}{p(x)}(s_2\eta)^{(p(x)-1)\frac{p(x)}{p(x)-1}}=\\
&=&\frac{1}{p(x)\eta^{p(x)(p(x)-1)}}\cdot s_1^{p(x)}+ \frac{p(x)-1}{p(x)} \eta^{p(x)}\cdot s_2^{p(x)}.\\
\end{eqnarray*}
Now it suffices to substitute $\tau(x)=\eta(x)^{p(x)}$. 
\end{proof}

\begin{lem}\label{crit}
Let $u\in W^{1,1}_{loc}(\Om)$ be defined everywhere by  the formula (see
e.g.~\cite{hajpunkt})
\begin{equation*}
u(x):= \limsup_{r\to 0}\barint_{B(x,r)} u(y)dy
\end{equation*}
and let $t\in \R$. Then
\[
\{ x\in \rn : u(x)=t \}\subseteq \{ x\in\rn :\nabla u(x)=0\}\cup N,
\]
where $N$ is a set of Lebesgue's measure zero.
\end{lem}

\bigskip

\noindent The main goal of this section is the following result. 
\begin{theo}[Caccioppoli estimate]\label{cac} Let $\Omega\subseteq\mathbb{R}^n$ be an open subset. 
We suppose that the measurable function $p:\Omega \rightarrow (1,\infty)$ satisfy~\eqref{P}, nonnegative $u
\in W^{1,p(x)}_{loc}(\Om)$ and $\Phi\in L^{1}_{loc}(\Om)$ satisfy PDI $-\Delta_{p(x)} u\geqslant \Phi$, in the sense of Definition~\ref{defnier}.
Assume further that functions $u$, $\Phi$, $p(x)$, $\s(x)$ and a parameter $\beta>0$ satisfy crucial conditions~\eqref{sx}
and~\eqref{sbeta}. 

 Then the~inequality
\begin{equation}\int_{\Omega}
\left(\Phi\cdot u+\s(x)|\nabla u|^{p(x)}\right)
{u}^{-\be-1} \chi_{\{u>0\}} \cdot \phi\,dx \leqslant \nonumber
\end{equation}
\begin{equation}
\label{caccc}\leqslant \int_{\Omega}
\frac{{(p(x)-1)}^{p(x)-1}}{{(p(x))}^{p(x)}(\be-\s(x))^{p(x)-1}}
\,u^{p(x)-\be-1}\chi_{\{\nabla u\neq 0\}}
\cdot|\nabla\phi|^{p(x)}\phi^{1-p(x)}\,dx,
\end{equation}
holds for every nonnegative Lipschitz function $\phi$ with
compact support in $\Om$ such that the integral $\int_{\sph}{|\nabla\phi|}^{p(x)}\phi^{1-p(x)}\, dx$ is finite.
\end{theo}

We call \eqref{caccc} Caccioppoli estimate,  because it involves $\nabla u$ on the left--hand side and, when we estimate $\chi_{\{\nabla u\neq 0\}}\leqslant 1$ on the right--hand side, then the  right--hand side depends only on $u$ (see e.g. \cite{cac,iwsbo}).

We note that we do not assume that the right--hand side in (\ref{caccc}) is finite.

The proof is based on the idea of the proof of Theorem~3.1 from~\cite{plap} whose further inspiration is the proof of Proposition~3.1 from~\cite{nonex}.

\begin{proof}[\textbf{Proof of Theorem \ref{cac}}] The proof follows by three steps.

\subsubsection*{Step 1. Derivation of a local inequality.}
We obtain the following lemma.
\begin{lem}\label{cacc1}We suppose that the measurable function $p:\Omega \rightarrow (1,\infty)$ satisfy~\eqref{P}, nonnegative $u
\in W^{1,p(x)}_{loc}(\Om)$ and $\Phi\in L^{1}_{loc}(\Om)$ satisfy PDI $-\Delta_{p(x)} u\geqslant \Phi$, in the sense of Definition~\ref{defnier}. 
 Assume further that $\be>0$ is arbitrary number and $\ve(x)$ is a bounded function with values separated from 0.

  Then, for every $0<\delta <R$, the inequality
\begin{eqnarray}\label{cacccs1}
 \hspace{-0.5cm}&\int_{\Omega }\Big(\Phi\cdot (u+\delta)+\big(\be - \frac{p(x)-1}{p(x)} \varepsilon(x) \big)|\nabla u|^{p(x)} \Big)
 (u+\delta)^{-\beta-1} \chi_{\{u\leqslant R-\delta \}} \cdot\phi~dx \quad&
 \\
	\hspace{-0.5cm}&\leqslant
\int_{\Omega } \frac{1}{p(x)\varepsilon(x)^{{p(x)-1}}}(u+\delta)^{p(x)-\be-1}
\chi_{\{ \nabla u\neq 0,\,
u\leqslant  R-\delta\}}\cdot|\nabla\phi|^{p(x)}\phi^{1-p(x)}\,dx +
C(\delta,R),&\nonumber
\end{eqnarray}
where \begin{eqnarray}
C(\delta,R)=&R^{-\be}\left[\int_{\Omega }|\nabla u|^{p(x)-2}\langle \nabla
u, \nabla {\phi}\rangle \chi_{
\{ \nabla u\neq 0,\, u> R-\delta \}} dx- \int_{\Omega}
\Phi \chi_{\{u>R-\delta\}} \phi dx\right]&\nonumber\\
\label{cdr}
\end{eqnarray}
 holds for every nonnegative Lipschitz function $\phi$ with
compact support in $\Om$.
\end{lem}

\noindent
\begin{proof}[\textbf{ Proof of Lemma \ref{cacc1}}]

We take $w=G$ (see \eqref{G}) in the left side of the inequality \eqref{nikfo} and note that
\begin{eqnarray}
 \label{111}
L&:=& \int_\Om \Phi \cdot G\,dx = \int_\Om
\Phi\cdot (u_{\delta,R})^{-\be}\phi\,dx=\\
&=& \int_{\Oc\{u\leqslant R-\delta\}} \Phi \cdot(u+\delta)^{-\be}\phi
\,dx+
R^{-\be}\int_{\Oc\{u>R-\delta\}}\Phi \cdot\phi\,dx.\nonumber
\end{eqnarray}

On the other hand, inequality \eqref{nikfo} implies
\begin{eqnarray*}
L&:=& \int_{\Om} \Phi \cdot G\,dx \leqslant
\langle -\Delta_{p(x)} u, G\rangle =\int_{\Oc\{\nabla u\neq 0\}} |\nabla u|^{p(x)-2}\langle \nabla u, \nabla G\rangle\,  dx=\\
&\stackrel{}{=}& -\be\int_{\Oc\{ \nabla u\neq 0,\, u\leqslant
R-\delta\} } |\nabla u|^{p(x)}(u +\delta)^{-\be-1}\phi\,
dx +\\
&&+\int_{ \Oc\{ \nabla u\neq 0,\, u\leqslant R-\delta \} }|\nabla
u|^{p(x)-2}\langle \nabla u, \nabla {\phi}\rangle
(u+\delta)^{-\be}\,  dx +\\
&&+R^{-\be}\int_{\Oc \{ \nabla u\neq 0,\, u> R-\delta
\} }|\nabla u|^{p(x)-2}\langle \nabla u, \nabla {\phi}\rangle \,  dx.
\end{eqnarray*}
Note that all the above integrals are finite, what follows from Lemma~\ref{nonex1} (for $0 \leqslant u \leqslant R-\delta$ we have $\delta\leqslant u+\delta\leqslant R$). We compute further that
\begin{eqnarray*}
&&\int_{ \Oc\{ \nabla u\neq 0,\, u\leqslant R-\delta \} }|\nabla
u|^{p(x)-2}\langle \nabla u, \nabla {\phi}\rangle
(u+\delta)^{-\be}\,  dx \leqslant \\
&&\leqslant \int_{ \Oc\{
\nabla u\neq 0,\, u\leqslant R-\delta \}} | \nabla u|^{p(x)-1} |
\nabla \phi |{ (u+\delta)}^{-\be}\, dx=\\
&&=\int_{\sph\cap \{ \nabla u\neq 0,\, u\leqslant R-\delta \}}
 \Big(
\frac{|\nabla\phi|}{\phi} {(u+\delta)} \Big)\cdot | \nabla u|^{p(x)-1}
(u+\delta)^{-\be-1}\,\phi \ dx.
 \end{eqnarray*} We  apply  Lemma \ref{lemmjednor} with $s_1=\frac{|\nabla\phi|}{\phi}
{(u+\delta)}$,
$s_2=|\nabla u|$ and an arbitrary bounded and continuous function $\tau(x)=\varepsilon(x)>0$ with values separated from $0$,  to get
\begin{eqnarray*}
&&\int_{ \Oc\{ \nabla u\neq 0,\, u\leqslant R-\delta \} }|\nabla
u|^{p(x)-2}\langle \nabla u, \nabla {\phi}\rangle
(u+\delta)^{-\be}\,  dx \leqslant \\ 
&&\leqslant \int_{ \sph\cap \{  \nabla u\neq
0,\, u\leqslant R-\delta \}} \frac{p(x)-1}{p(x)} \varepsilon(x) |\nabla u|^{p(x)} (u+\delta)^{-\be-1}\phi \, dx+
\\&&+\int_{ \sph\cap \{ \nabla u\neq 0,\, u\leqslant R-\delta \}} \frac{1}{p(x)\varepsilon(x)^{p(x)-1}}
\Big(\frac{|\nabla\phi|}{\phi}\Big)^{p(x)}
(u+\delta)^{p(x)-\be-1}\phi\,  dx.
\end{eqnarray*}
Combining these estimates we {deduce} that
\begin{eqnarray*}
L&\leqslant&
\int_{\Oc\{ \nabla u\neq 0,\, u\leqslant
R-\delta\} }  \big(-\be+ \frac{p(x)-1}{p(x)}\varepsilon(x)\big)|\nabla u|^{p(x)}(u +\delta)^{-\be-1}\phi\,
dx +\\
&&+\int_{ \sph\cap \{ \nabla u\neq 0,\, u\leqslant R-\delta \}} \frac{1}{p(x)\varepsilon(x)^{p(x)-1}}
(u+\delta)^{p(x)-\be-1} |\nabla \phi|^{p(x)}\phi^{1-p(x)} \,  dx  +\\
&&+R^{-\be}\int_{\Oc \{ \nabla u\neq 0,\, u> R-\delta
\} }|\nabla u|^{p(x)-2}\langle \nabla u, \nabla {\phi}\rangle \,  dx.
\end{eqnarray*}
This and (\ref{111})
imply
 \begin{eqnarray*}
&&\int_{\Oc\{u\leqslant R-\delta\}} \Phi \cdot(u+\delta)^{-\be}\phi
\,dx+\\
&&+\int_{\Oc\{ \nabla u\neq 0,\, u\leqslant
R-\delta\} } \big(\be-\frac{p(x)-1}{p(x)}\varepsilon (x)\big)|\nabla u|^{p(x)}(u +\delta)^{-\be-1}\phi\,
dx \leqslant\\
&\leqslant&\int_{ \sph\cap \{ \nabla u\neq 0,\, u\leqslant R-\delta \}} \frac{1}{p(x)\varepsilon(x)^{p(x)-1}}
(u+\delta)^{p(x)-\be-1} |\nabla \phi|^{p(x)}\phi^{1-p(x)} \,  dx  +
C(\delta,R),
\end{eqnarray*}
 where $C(\delta,R)$ is given by \eqref{cdr}.\end{proof}

\begin{rem}\rm
Introduction of parameters $\delta$ and $R$ was necessary as we
needed to move some finite quantities in the estimates to opposite sides
of inequalities. 
\end{rem}

\subsubsection*{Step 2. Passing to the limit with $\delta\searrow 0$.}
We show that when $\be>0$ is an arbitrary number,  $\varepsilon(x)$ is a bounded function with values separated from $0$, such that $\be - \frac{p(x)-1}{p(x)} \varepsilon (x)=:\s(x)$,  then
 for any $R>0$
\begin{eqnarray}
\label{localcaccc}
&\int_{\Omega} \left(\Phi \cdot u
+\s(x)|\nabla
u|^{p(x)}\right) u^{-\be-1}
\chi_{\{0<u\leqslant R\}} \cdot
\phi\, dx& \\
&\leqslant
\int_{\Omega} \frac{1}{p(x)\varepsilon(x)^{p(x)-1}} u^{p(x)-\be-1}\chi_{\{\nabla u\neq 0,u\leqslant
R\}} \cdot|\nabla\phi|^{p(x)}\phi^{1-p(x)}\,dx
+C(R),&\nonumber
\end{eqnarray}
where
\begin{eqnarray*}
C(R)=& R^{-\be}\Big[\big|\int_{\Omega} |\nabla u|^{p(x)-2}| \nabla u|\chi_{\{
u\geqslant \frac{R}{2}\} }\cdot | \nabla
{\phi}| \, dx\big|+ \int_{\Omega}
\Phi \chi_{\{u\geqslant \frac{R}{2}\}} \cdot \phi\,dx\Big]&
\end{eqnarray*}
 holds for every nonnegative Lipschitz function $\phi$ with
compact support in $\Om$ such that the integral $\int_{\sph\cap \{\nabla u\neq
0\}}|\nabla\phi|^{p(x)}\phi^{1-p(x)}\,dx$ is finite.
Moreover, all  quantities appearing in
(\ref{localcaccc}) are finite.

We show first that under our assumptions, when
 $\delta \searrow 0$, we have \begin{eqnarray}\label{deltadozera}
&\int_{\Omega } \frac{1}{p(x)\varepsilon(x)^{p(x)-1}} (u+\delta)^{p(x)-\be-1}\chi_{\{ \nabla u\neq 0,\,u+\delta\leqslant R\}}\cdot|\nabla\phi|^{p(x)}\phi^{1-p(x)}\,dx\to\\
&\to\int_{\Omega}\frac{1}{p(x) \varepsilon(x)^{p(x)-1}} u^{p(x)-\be-1}\chi_{\{ \nabla u\neq 0,\,u\leqslant R\} }\cdot|\nabla\phi|^{p(x)}\phi^{1-p(x)}\,dx\nonumber
\end{eqnarray}
for every nonnegative Lipschitz function $\phi$ with
compact support in $\Om$ such that the integral $\int_{\sph\cap \{\nabla u\neq
0\}}|\nabla\phi|^{p(x)}\phi^{1-p(x)}\, dx$ is finite.

We note that $(u+\delta)^{p(x)-\be-1}\chi_{\{u+\delta\leqslant R\}}\stackrel{\delta\to 0}{\to}  u^{p(x)-\be-1}\chi_{\{u\leqslant R\}}$ a.e.  This follows from Lemma~\ref{crit} (which
gives that the set $\{u=0,\ |\nabla u| \neq 0\}$ is of measure zero) and the continuity outside
zero of the involved functions.

We show \eqref{deltadozera} independently on separate subsets of domains of integration. Hence, we have
 \begin{eqnarray*}
&\int_{\Omega\cap \{\nabla u\neq 0\}}\frac{1}{p(x)\varepsilon(x)^{p(x)-1}}(u+\delta)^{p(x)-\be-1}\chi_{\{ u+\delta\leqslant R\}}\cdot|\nabla\phi|^{p(x)}\phi^{1-p(x)}\,dx=&\\
&=\sum\limits_{i=1}^3\int_{E_{i} \cap \{\nabla u\neq 0\}}\frac{1}{p(x)\varepsilon(x)^{p(x)-1}} (u+\delta)^{p(x)-\be-1}\chi_{\{ u+\delta\leqslant R\}}\cdot|\nabla\phi|^{p(x)}\phi^{1-p(x)}\,dx,&
\end{eqnarray*}
 where
\begin{eqnarray*}
&E_{1}=\left\{ x \in \Omega: p(x)-\be-1=0 \right\},&\\
&E_{2}=\left\{x \in \Omega: p(x)-\be-1<0 \right\},&\\
&E_{3}=\left\{ x \in \Omega: p(x)-\be-1>0 \right\}.&
\end{eqnarray*}

Convergence on $E_1$ follows from the Lebesgue's Monotone Convergence Theorem, as on this set the only expression involving $\delta$ is the characteristic function $\chi_{\{u+\delta\leqslant R\}}$.

Let us concentrate on the case when  $\delta \searrow 0$ on $E_2$. We apply the Lebesgue's Monotone Convergence Theorem as on this set
\[(u+\delta)^{p(x)-\be-1}\chi_{\{u+\delta\leqslant R\}}\nearrow  u^{p(x)-\be-1}\chi_{\{u\leqslant R\}}.
\] Indeed, we note first that then for a.e. $x \in \Omega$ such that $u(x)>0$ we have that $u+\delta \searrow u$. Hence, also $(u+\delta)^{p(x)-\be-1}\nearrow u^{p(x)-\be-1}\neq 0$. Secondly, we observe that then
 for a.e. $x \in \Omega$ we have $\chi_{\{0\leqslant u\leqslant R-\delta\}} \leqslant \chi_{\{0 < u\leqslant R-\delta\}} \nearrow \chi_{\{0<u<R\}}$.

In the case of $E_3$, without loss of generality, we assume that $R>1$. Then we apply the Lebesgue's Dominated Convergence Theorem as
\begin{eqnarray*}
&\int_{E_3 \cap \{ \nabla u \neq 0\}} \frac{1}{p(x) \varepsilon(x)^{p(x)-1}}(u+\delta)^{p(x)-\be-1}
\chi_{\{u+\delta\leqslant R\}}
\cdot|\nabla\phi|^{p(x)}\phi^{1-p(x)}\,dx \leqslant&\\
&\leqslant R^{{p^+}-\be-1}\frac{\widetilde{\varepsilon}}{p^-} \int_{E_3 \cap \{ \nabla u \neq 0\}} \chi_{\{u \leqslant R\}}\cdot|\nabla\phi|^{p(x)}\phi^{1-p(x)}\,dx<\infty,&
\end{eqnarray*}
where $\widetilde{\varepsilon}=\sup_{x\in\Om}\left[\varepsilon(x)^{1-p(x)}\right].$ The details are left to the reader.

To complete the proof of Step 2 we note that \eqref{deltadozera} says that, when $\delta\searrow 0$, the first integral on the right--hand side of  \eqref{cacccs1} is convergent to the first integral of the right--hand side of \eqref{localcaccc}. To deal with the second expression note that for $\delta\leqslant\frac{R}{2}$, we
have
\begin{eqnarray*}
&|C(\delta,R)| \leqslant
\Big|R^{-\be}\int_{\Omega}|\nabla u|^{p(x)-2}\langle \nabla u, \nabla {\phi}\rangle \chi_{\{u> R-\delta\}}\,  dx \Big| +&\\
&\quad\quad+\Big| R^{-\be}\int_{\Omega}\Phi \chi_{\{u>R-\delta\}}\cdot\phi\,dx \Big|\leqslant
 C(R).&
\end{eqnarray*}

 It suffices now to pass to the limit with $\delta\searrow 0$ on the left--hand side of~\eqref{cacccs1}. We do it due to the Lebesgue's Monotone Convergence Theorem as the expression in brackets is nonnegative
and decreasing. Indeed, the condition \eqref{sx} implies
\begin{equation*}
\Phi\cdot (u+\delta)+\s (x)|\nabla u|^{p(x)} \geqslant\Phi\cdot u+\s (x)|\nabla u|^{p(x)}
\geqslant 0\quad \mathrm{a.e. \; on } \quad \Oc \{u>0\}.
\end{equation*}

\subsubsection*{ Step 3. We let
$R\to\infty$  and finish the proof.}
 Without loss of generality we can assume that the
integral in the right--hand side of \eqref{caccc} is finite, as
otherwise the inequality follows trivially.  Note  that since $|\nabla
u|^{p(x)-2}\langle\nabla u,\nabla\phi\rangle$ and $\Phi\phi$ are integrable  we have $\lim_{R\to \infty}
C(R)=0$.  Therefore, \eqref{caccc} follows from \eqref{localcaccc} by the Lebesgue's Monotone
Convergence Theorem (note that $\varepsilon(x)=\frac{p(x) (\be-\s (x))}{p(x)-1}$ by the choice of $\s(x)$).
\end{proof}

\section{General $p(x)$--Hardy inequality}\label{SecHar}

In the proof of $p(x)$--Hardy inequality we need the following lemma.

\begin{lem} \label{inequality2}
Let $p:\Omega \to (1,\infty)$ satisfy~\eqref{P} and $s_1, s_2 \geqslant 0$, then the following inequality holds for a.e. $x\in\Om$
 \begin{equation}
 \label{eqelem}
(s_1+s_2)^{p(x)}\leqslant 2^{(p(x)-1)\chi_{\{s_1\neq 0\}}} \left(s_1^{p(x)} + s_2^{p(x)}\right).
 \end{equation}
\end{lem}

\begin{rem}\label{remeqelem}\rm
Note that in this lemma the role of $s_1$ is not the same as $s_2.$ If $s_1=0$, then \eqref{eqelem} becomes $s_2^{p(x)}=s_2^{p(x)}$. This is necessary to retrieve Theorem~4.1 from \cite{plap} (concerning constant exponent case) with the best constant via our investigations (see Theorem~\ref{theoplap} here).
\end{rem}

\bigskip

\noindent Now we state our main result.
\begin{theo} [$p(x)$--Hardy inequality]
\label{theoplapx} Let $\Omega\subseteq\mathbb{R}^n$ be an open subset not necessarily bounded and $p\in{\cal{P}}(\Omega)$.
Let nonnegative $u
\in W^{1,p(x)}_{loc}(\Om)$ and $\Phi\in L^{1}_{loc}(\Om)$ satisfy PDI $-\Delta_{p(x)} u\geqslant \Phi$, in the sense of Definition~\ref{defnier}.
Assume further that functions $u$, $\Phi$, $p(x)$, $\s(x)$ and a parameter $\beta>0$ satisfy crucial conditions~\eqref{sx}
and~\eqref{sbeta}. 

Then for every Lipschitz function $\xi$ with compact support in $\Om$ we have
\begin{equation}\label{hardypx}
\int_\Om \ |\xi|^{p(x)} \mu_{1,\beta}(dx)\leqslant \int_\Om |\nabla \xi|^{p(x)}\mu_{2,\beta}(dx)+\int_\Om \left|\xi {\log \xi } \right|^{p(x)}\cdot \frac{\left|\nabla p(x)\right|^{p(x)}}{{p(x)}^{p(x)}} \mu_{2,\beta}(dx),
\end{equation}
where
\begin{eqnarray}\label{mu1px}
&\mu_{1,\beta}(dx)&=\big(\Phi\cdot u+\sigma(x) |\nabla u|^{p(x)}\big)\cdot u^{-\be-1}\chi_{\{u>0\}}\ dx,
\\\label{mu2px}
&\mu_{2,\beta}(dx)&=   { \Big(\frac{p(x)-1}{\be-\s(x)}\Big)^{p(x)-1} } 2^{(p(x)-1)\chi_{\left\{|\nabla p|\neq 0\right\}}}u^{p(x)-\be-1}\chi_{\{|\nabla u|\neq 0\}}\ dx.\end{eqnarray}
\end{theo}

\begin{proof}
We are going to apply Theorem \ref{cac} and, after substituting a certain form of function $\phi$, we estimate the right--hand side of \eqref{caccc}.

We take  $\xi (x)=(\phi(x))^\frac{1}{p(x)}$. Then whenever $\phi>0$, we have
\[\nabla \xi=\frac{1}{p(x)} \phi^{\frac{1}{p(x)}-1}\nabla \phi-\frac{\log \phi}{p^2(x)} \phi^{\frac{1}{p(x)}}\nabla p(x).\]
Equivalently, we have
\begin{equation}\label{phiphi}
\phi^{\frac{1}{p(x)}-1}\nabla \phi=p(x)\nabla \xi+\frac{\log \phi}{p (x)} \phi^{\frac{1}{p(x)}}\nabla p(x).
\end{equation}
We observe that
\[\left\{\frac{\log \phi}{p (x)} \phi^{\frac{1}{p(x)}}\left|\nabla p(x)\right|\neq 0\right\}\subseteq \left\{|\nabla p(x)|\neq 0\right\}=: P.\]
We apply Lemma~\ref{inequality2} to \eqref{phiphi} (with $s_1=\frac{\log \phi}{p (x)} \phi^{\frac{1}{p(x)}}\left|\nabla p(x)\right|$ and $s_2=p(x)\nabla \xi$) to get
\begin{eqnarray}\nonumber\left|  \phi^{\frac{1}{p(x)}-1}\nabla \phi\right|^{p(x)}=\left|p(x)\nabla \xi+\frac{\log \phi}{p (x)} \phi^{\frac{1}{p(x)}}\nabla p(x)\right|^{p(x)}\leqslant\\
\leqslant  2^{(p(x)-1)\chi_P}\left|p(x)\nabla \xi\right|^{p(x)}+  2^{(p(x)-1)\chi_P}\left|\frac{\log \phi}{p(x)} \phi^{\frac{1}{p(x)}}\nabla p(x)\right|^{p(x)} .\label{V}\end{eqnarray}
We substitute  $\xi^{p(x)}=\phi$ on the right--hand side of \eqref{V} to obtain
\begin{eqnarray}
&\left|\nabla \phi\right|^{p(x)}\phi^{1-{p(x)}}= \nonumber\left|\phi^{\frac{1}{p(x)}-1}\nabla \phi\right|^{p(x)}\leqslant \\
\nonumber&\leqslant  2^{(p(x)-1)\chi_P}\left|p(x)\nabla \xi \right|^{p(x)}+  2^{(p(x)-1)\chi_P}\left|\frac{\log (\xi^{p(x)})}{p(x)} \xi\nabla p(x)\right|^{p(x)} =\\
&=   2^{(p(x)-1)\chi_P}\left|p(x)\nabla \xi\right|^{p(x)}+  2^{(p(x)-1)\chi_P}\left|\xi{\log \xi} \nabla p(x)\right|^{p(x)} .\label{poszac}
\end{eqnarray}
We recall that $\mu_{1,\beta}$  is given in \eqref{mu1px} and let us denote $\mu$ as follows
\[\mu(dx)= \frac{(p(x)-1)^{p(x)-1}}{{p(x)}^{p(x)} (\be-\s (x))^{p(x)-1}} u^{p(x)-\be-1}\chi_{\{|\nabla u|\neq 0\}}\ dx.\]
Applying \eqref{poszac}, we get   \begin{eqnarray*}
&\int_{\Om}\  |\nabla\phi|^{p(x)}\phi^{1-{p(x)}}\ \mu(dx)= \int_{\Om} \,\left|\phi^{\frac{1}{p(x)}-1}\nabla \phi\right|^{p(x)}\,\mu(dx) \leqslant\\
&\leqslant \int_{\Om}\  2^{(p(x)-1)\chi_P}\left(\left|p(x)\nabla \xi\right|^{p(x)} + \left|\xi{\log \xi} \nabla p(x)\right|^{p(x)} \right) \mu(dx) =\\
&=\int_{\Om} \left|\nabla \xi\right|^{p(x)}\mu_{2,\beta}(dx)+ \int_{\Om}\left|\xi{\log \xi} \right|^{p(x)}  \frac{|\nabla p(x)|^{p(x)}}{p(x)^{p(x)}} \, \mu_{2,\beta}(dx),
\end{eqnarray*}
 where $\mu_{2,\beta}(dx)$ is given by \eqref{mu2px}.

  Summing up, by Theorem \ref{cac}, we obtain \begin{eqnarray*}&\int_{\Om} \xi^{p(x)} \mu_{1,\beta}(dx) =\int_{\Om} \phi\ \mu_{1,\beta}(dx) \leqslant \int_{\Om}  |\nabla\phi|^{p(x)}\phi^{1-{p(x)}} \ \mu(dx) \leqslant&\\
&\leqslant \int_{\Om} \left|\nabla \xi\right|^{p(x)}\mu_{2,\beta}(dx)+ \int_{\Om}\left|\xi{\log \xi} \right|^{p(x)} \frac{\left|\nabla p(x)\right|^{p(x)}}{{p(x)}^{p(x)}} \mu_{2,\beta}(dx),&\end{eqnarray*}
 which completes the proof.
\end{proof}

\section{One--dimensional case}\label{Sec1d}

This section is devoted to the case when $\Omega=I\subseteq\R$ is an open interval (not necessarily finite). We give here a few original examples indicating that our conditions on admissible functions $p(x)$ are not very restrictive.

Let us start with the following direct corollary of  Theorem~\ref{theoplapx}.

\begin{coro}\label{coropx4} Suppose $I\subseteq(-M,M)\subseteq\R$, with some $M>0$,  is a bounded open subset, $u=M-|x|$, and $p \in \mathcal{P}(I)$. Assume further that nonnegative $\s(x)$ and $\beta>0$ satisfy crucial condition~\eqref{sbeta}.

Then for every Lipschitz function $\xi$ with compact support in $I$, we have
\begin{equation*}
\int_{I}\ |\xi|^{p(x)} \mu_{1,\beta}(dx)\leqslant \int_{I}| \xi'|^{p(x)}\mu_{2,\beta}(dx)+\int_{I} \left|\xi{\log \xi} \right|^{p(x)} \frac{|p'(x)|^{p(x)}}{p(x)^{p(x)}}\ \mu_{2,\beta}(dx),
\end{equation*}
where
\begin{eqnarray*}
&\mu_{1,\beta}(dx)&= (M-|x|)^{-\beta-1}\sigma(x)\ dx,
\\\label{mu2px4}
&\mu_{2,\beta}(dx)&=  (M-|x|)^{p(x)-\beta-1} \left[2 \cdot \frac{p(x)-1}{\beta-\sigma(x)}\right]^{p(x)-1} \ dx.
\end{eqnarray*}
\end{coro}
\begin{proof} We apply  $u=M-|x|$ on $\Omega=I=(-M,M)$ in Theorem~\ref{theoplapx}. In this case $u'=-\,sgn(x)$, $u''\equiv 0$ outside $0$ and $u''(0)=-2\delta_0$ is $-2$ times Dirac delta. It enables to choose $\Phi\equiv 0$ in the PDI $-\Delta_{p(x)}u\geqslant \Phi$ due to Definition~\ref{defnier}. Direct computations finish the proof.
\end{proof}

When the considered solution to nonlinear problem is more regular, one--dimesional version of~Theorem~\ref{theoplapx} can be reduced in the following way.
\begin{theo}[One--dimensional inequality]
\label{theo-one-dim}

Let $I \subseteq \R$, $p \in \mathcal{P}(I)$, and $u
\in W^{1,p(x)}_{loc}(I)\cap W^{2,1}_{loc}(I)$ be a nonnegative function, such that $|u'|^{p(x)-2}u'\in W^{1,1}_{loc}(I)$. Assume further that $\s(x)$ and $\beta>0$ satisfy crucial condition~\eqref{sbeta} and the following condition is satisfied
\begin{equation}
\label{g-one-dim} g(x):=\sigma(x) (u')^2-p'(x) uu'\log |u'|-(p(x)-1) uu'' \geqslant 0 \quad a.e.\ x\in I.
\end{equation}

Then, for every Lipschitz function $\xi$ with compact support in $I$, we have
\begin{equation}\label{hardypx2}
\int_I \ |\xi|^{p(x)} \mu_{1,\beta}(dx)\leqslant \int_I |\xi'|^{p(x)}\mu_{2,\beta}(dx)+\int_I \left|\xi{\log \xi} \right|^{p(x)}   \frac{\left|p'(x)\right|^{p(x)}}{{p(x)}^{p(x)}}\,\mu_{2,\beta}(dx),
\end{equation}
where
\begin{eqnarray*}
&\mu_{1,\beta}(dx)&=\frac{|u'|^{p(x)-2}}{u^{\be+1}} g(x) \chi_{\{u>0\}} dx,
\\
&\mu_{2,\beta}(dx)&=    \left(\frac{p(x)-1}{\be-\s(x)}\right)^{p(x)-1}  2^{(p(x)-1)\chi_{\{p'\neq 0\}}}u^{p(x)-\be-1}\chi_{\{| u'|\neq 0\}}\ dx.
\end{eqnarray*}
\end{theo}
\begin{proof}
It suffices to apply Theorem~\ref{theoplapx} with $u=u(x)$, $x\in I$. Suppose $\widetilde{I}$ is the set where $u''$ is well defined, then
\begin{eqnarray*}
\Delta_{p(x)} u=(|u'|^{p(x)-2}u')'=(|u'|^{p(x)-2})'u'+|u'|^{p(x)-2}u''\quad \mathrm{on} \quad \widetilde{I}
\end{eqnarray*}
and thus
\begin{eqnarray*}
-\Delta_{p(x)} u=-|u'|^{p(x)-2} \left[p'(x) \cdot u'\log |u'|+(p(x)-2) u''  +u''\right]\ \mathrm{on} \ \widetilde{I}.
\end{eqnarray*}
We set
\begin{equation*}
\Phi=\left\{\begin{array}{lcl}
-\Delta_{p(x)} u&\mathrm{if}&u\in\widetilde{I},\\
0&\mathrm{if}&u\in I \setminus\widetilde{I},
\end{array}\right.
\end{equation*}
which satisfies all the restrictions of Theorem~\ref{theoplapx}.
Direct computations gives inequality \eqref{hardypx2}.
\end{proof}

We illustrate the above theorem by several examples. We give below the inequality with power--type weights, where we allow  $I=(0,\infty)$.
\begin{coro}\label{coropx45} Suppose $I \subseteq\R_+$ is an open subset,  $p \in \mathcal{P}(I)$, and $\alpha\in\R$ is an arbitrary number.
Assume further that $\s(x)$ and $\beta>0$ satisfy crucial condition~\eqref{sbeta} and the following condition is satisfied
 \begin{equation}
\label{g(x)kol} \overline{g}(x):= \sigma(x)\alpha^2 -  p'(x) x \alpha\log|\alpha x^{\alpha-1}|+(p(x)-1) \alpha(1-\alpha)\geqslant 0\ \ a.e.\ x\in I.
\end{equation}

Then, for every Lipschitz function $\xi$ with compact support in $ I$, we have
\begin{equation}\label{hardypx8}
\int_{I}\ |\xi|^{p(x)} \mu_{1,\beta}(dx)\leqslant \int_{I}| \xi'|^{p(x)}\mu_{2,\beta}(dx)+\int_{I} \left|\xi{\log \xi} \right|^{p(x)} \frac{|p'(x)|^{p(x)}}{p(x)^{p(x)}}\ \mu_{2,\beta}(dx),
\end{equation}
where
\begin{eqnarray*}\label{mu1px8}
&\mu_{1,\beta}(dx)&= | \alpha|^{p(x)-2} x^{\alpha (p(x)-\beta-1)-p(x)} \cdot \overline{g}(x) \ dx,
\\\label{mu2px8}
&\mu_{2,\beta}(dx)&=   x^{\alpha (p(x)-\beta-1)} \left(2 \cdot \frac{p(x)-1}{\beta-\sigma(x)}\right)^{p(x)-1} \ dx.
\end{eqnarray*}
\end{coro}

\begin{proof}
 We apply Theorem \ref{theo-one-dim} with the function $u=x^{\alpha}$.
We note that $u'=\alpha x^{\alpha-1}$ and $u''=\alpha (\alpha-1) x^{\alpha-2}$ and thus  according to \eqref{g-one-dim} we have
\[g(x)=x^{2\alpha-2} \Big[\s (x)\alpha^2-p'(x) x \alpha \log |\alpha x^{\alpha-1}| +(p(x)-1)\alpha(1-\alpha)\Big],\]
which is nonnegative due to \eqref{g(x)kol}.
 Direct computations gives inequality \eqref{hardypx8} with the desired measures.
\end{proof}

 As an another example we give the following inequality, where we allow  $I=(0,\infty)$.
\begin{coro}\label{coropx454} Suppose $I \subseteq\R_+$ is an open subset,  $p \in \mathcal{P}(I)$, and $a>0$ is an arbitrary number.
Assume further that $\s(x)$ and $\beta>0$ satisfy condition~\eqref{sbeta} and the following condition is satisfied\[\overline{g}(x):= \sigma(x) + p'(x)x   \log \frac{a}{x^2} -2p(x)+2 \geqslant 0\quad a.e.\ in\ I.\]

Then for every Lipschitz function $\xi$ with compact support in $I$, we have
\begin{equation}\label{hardypx9}
\int_{I}\ |\xi|^{p(x)} \mu_{1,\beta}(dx)\leqslant \int_{I}| \xi' |^{p(x)}\mu_{2,\beta}(dx)+\int_{I} \left|\xi  \log \xi \right|^{p(x)} \frac{|p'(x)|^{p(x)}}{p(x)^{p(x)}}\ \mu_{2,\beta}(dx),
\end{equation}
where
\begin{eqnarray*}\label{mu1px9}
&\mu_{1,\beta}(dx)&= \Big(\frac{a}{x}\Big)^{ p(x)-\beta-1} x^{-p(x)} \cdot \overline{g}(x) \ dx,
\\\label{mu2px9}
&\mu_{2,\beta}(dx)&=   \Big(\frac{a}{x}\Big)^{ p(x)-\beta-1} \left(2 \cdot \frac{p(x)-1}{\beta-\sigma(x)}\right)^{p(x)-1} \ dx.
\end{eqnarray*}
\end{coro}

\begin{proof}
 We apply Theorem \ref{theo-one-dim} with the function $u=\frac{a}{x}$.
 Direct computations gives inequality \eqref{hardypx9} with the desired measures.
\end{proof}

  We obtain also an inequality with exponential--type weights, where we allow  $I=(0,\infty)$ as well as $I=(-\infty,\infty)$.

\begin{coro}\label{coropx3}  Suppose $I \subseteq\R$ is an open subset,  $p \in \mathcal{P}(I)$, and $a>0$ is an arbitrary number. Assume further that $\s(x)$ and $\beta>0$ satisfy condition~\eqref{sbeta} and the following condition  is satisfied \[\bar{g}(x):=\sigma(x)- p'(x) x -p(x)+1\quad a.e.\ in\ I.\]

Then for every Lipschitz function $\xi$ with compact support in $I$, we have
\begin{equation}\label{hardypx3}
\int_I\ |\xi|^{p(x)} \mu_{1,\beta}(dx)\leqslant \int_I | \xi'|^{p(x)}\mu_{2,\beta}(dx)+\int_I \left|\xi{\log \xi} \right|^{p(x)} \frac{|p'(x)|^{p(x)}}{p(x)^{p(x)}}\ \mu_{2,\beta}(dx),
\end{equation}
where
\begin{eqnarray*}\label{mu1px3}
&\mu_{1,\beta}(dx)&= \bar{g}(x) e^{x(p(x)-\beta -1)}\ dx,
\\\label{mu2px3}
&\mu_{2,\beta}(dx)&=   \left(2 \cdot \frac{p(x)-1}{\beta-\sigma(x)}\right)^{p(x)-1} e^{x(p(x)-\beta-1)} \ dx.
\end{eqnarray*}
\end{coro}

\begin{proof} We apply Theorem \ref{theo-one-dim} with the function $u=e^{x}$.
We note that according to \eqref{g-one-dim} we have
\[g(x)=e^{2x}\left(\s(x)- p'(x) x-p(x)+1\right),\]
which is nonnegative by the assumption.
 Direct computations gives inequality \eqref{hardypx3} with the desired measures.
\end{proof}
 \begin{rem}
We give examples of triplets of $p(x)$, the interval $I$, and $\s(x)$ admissible in Corollary~\ref{coropx3}.
\begin{itemize}
\item For arbitrary $d>0$, we may take  $p(x)= 1+\frac{d}{|x|+1}$, any interval ${I} \subseteq\mathbb{R}$, and any function $\s(x)$, which is nonnegative and continuous on $\overline{I}$.
\item We may take  $p(x)= e^{ x}$, any finite interval $I\subseteq\R_+$, and any function $\s(x)$ continuous on $\overline{I}$ such that \[\s(x)\geqslant  (x  +1)e^{ x}-1.\]
\item We may take $p(x)=2-e^{-x^2}$, $I=(0,\infty)$, and $\sigma(x)\geqslant e^{-x^2} (2x^2 -1)+1$\\ (e.g.~$\s(x)\equiv 2e^{-3/2}+1$).
\end{itemize}
\end{rem}

\section{Links with the existing results}\label{SecLink}

In this section we present several applications of Theorem~\ref{theoplapx}.
We start with re--obtaining the main result of Skrzypczak~\cite{plap}, which deals with constant function $p$ and  implies classical Hardy inequality with optimal constant (see \cite{plap}, Theorem~5.1). Then we concentrate on the comparison with the results of Harjulehto--H\"{a}st\"{o}--Koskenoja~\cite{HaHaKo} and Mashiyev--\c{C}eki\c{c}--Mamedov--Ogras~\cite{MCMO}. We mention also the related papers considering inequalities involving Hardy operator.

In our paper~\cite{barskrzy2} we focus on $n$--dimensional inequalities, in particular with radial weights.

\subsubsection*{ Results of Skrzypczak \cite{plap,bcp-plap}}

When we
consider $1<p(x)\equiv p<\infty$ in Theorem~\ref{theoplapx}, we retrieve the main result of  \cite{plap}, implying the classical Hardy inequality with the optimal constant (see \cite{plap} for the  details and the numerous other examples). Moreover, the following theorem leads to Hardy--Poincar\'{e} inequalities with the weights of a~type $\left(1+|x|^\frac{p}{p-1}\right)^\al$, where the constants are proven to be optimal  for sufficiently big parameter $\al>0$ (see \cite{bcp-plap} for the details).
\begin{coro}[{\cite[Theorem 4.1]{plap}}]
\label{theoplap}
Assume that $1<p <\infty$ and $u
\in W^{1,p}_{loc}(\Om)$   is a nonnegative solution to the PDI \  \(-\Delta_p u \geqslant \Phi\),  in the sense of~Definition  \ref{defnier}, where function $\Phi$ is locally integrable and satisfies the condition  \begin{equation}\label{s0}
\mathbf{(\Phi,p)}\quad\quad\s_0:=\inf \left\{\s\in\R: {\Phi\cdot u}+{\s |\nabla u|^p}\geqslant 0 \quad a.e.\ in\ \Omega\cap \{u>0\}\ \right\}\in\R.
\end{equation} Assume further that $\be$ and $\s$ are arbitrary numbers such that $\be>0$ and $\be>\s\geqslant\s_0$.
Then, for every Lipschitz function $\xi$ with compact support in $\Om$, we have
\[
\int_\Om \ |\xi|^p \mu_{1,\beta}(dx)\leqslant \int_\Om |\nabla \xi|^p\mu_{2,\beta}(dx),
\]
where
\begin{eqnarray*}\label{mu1p}
&\mu_{1,\beta}(dx)&= \left(\frac{\be-\s}{p-1}\right)^{p-1}\big(\Phi\cdot u+\sigma |\nabla u|^p\big)\cdot u^{-\be-1}\chi_{\{u>0\}}\ dx,
\\\label{mu2p}
&\mu_{2,\beta}(dx)&=  u^{p-\be-1}\chi_{\{|\nabla u|\neq 0\}}\ dx.
\end{eqnarray*}
\end{coro}

\begin{rem} \rm The paper~\cite{akiraj} applies the results of~\cite{plap} in order to obtain Poincar\'{e} inequalities with the best constants~\cite[Remark~7.6]{akiraj}. It is also proven therein that Hardy inequalities obtained in~\cite{plap} lead to the solvability of certain family of degenerated PDEs like ${\rm div}\left(\rho(x)|\nabla u(x)|^{p-2}\nabla u(x)\right)=x^*$, where $x^*$ is a functional on the weighted Sobolev space $W^{1,p}_\rho(\Omega)$,  involving degenerated $p$--Laplacian~\cite[Theorem 7.12]{akiraj}. Moreover, the results of~\cite{plap}  enable to formulate alternative interpretation of the  first eigenvalue of $p$--Laplacian~\cite[Remark~7.7]{akiraj}.
\end{rem}

\subsubsection*{ Results of Harjulehto--H\"{a}st\"{o}--Koskenoja \cite{HaHaKo}}

Paper \cite{HaHaKo} concerns the $n$--dimensional  norm version of Hardy--type inequality, but also the one--dimensional case is specially emphasized  therein.
 Let us mention the following result.

\begin{theo}[{\cite[Theorem 5.2]{HaHaKo}}] \label{theoHaHaKo1d}
Let $I=[0,M)$ for $M<\infty$, the variable exponent $p: I \rightarrow [1, \infty)$ be bounded, $p(0)>1$ and
\[
\limsup_{x \rightarrow 0^+} (p(x)-p(0))\log \frac{1}{x} <\infty.
\]
Moreover, suppose $\essinf_{x \in (0,x_0)} p(x)=p(0)$ for some $x_0 \in (0,1)$.

If $a \in[0, 1-\frac{1}{p(0)})$, then Hardy--type inequality
\begin{equation} \label{hhk2}
\|\xi(x) x^{a-1}\|_{L^{p(x)}(I)} \leqslant C \|\xi'(x)x^a\|_{L^{p(x)}(I)}
\end{equation}
 holds for every $\xi \in W^{1,p(x)} (I)$ with $\xi(0)=0$.
\end{theo}

 We have the following related result.
\begin{coro}\label{coropx1ha2} Suppose $I\subseteq\R_+$ is an open subset and $a,\beta>0$ are arbitrary numbers. Let $p \in \mathcal{P}(I)$ and assume  that there exists a~continuous function $A(x)$ such that \begin{equation}
\label{A} a\beta+(1-a) x p'(x) \log x+(a-3)(p(x)-1)\geqslant A(x)\geqslant 0.
\end{equation}
Then, for every Lipschitz function $\xi$ with compact support in $I$,  we have
\begin{equation}\label{hardypx1ha}
\int_{I}\ |x^{a-1}\xi|^{p(x)} \mu_{1,\beta}(dx)\leqslant \int_{I}|x^a \xi'|^{p(x)}\mu_{2,\beta}(dx)+\int_{I}\left( x^a |\xi {\log \xi|}  \frac{|p'(x)| }{p(x)}\right)^{p(x)}\; \mu_{2,\beta}(dx),
\end{equation}
where
\begin{eqnarray*}
&\mu_{1,\beta}(dx)&= {x^{-a(\beta+1)}}  A(x)\, dx,
\\
&\mu_{2,\beta}(dx)&=   {x^{-a(\beta+1)}} \, dx.
\end{eqnarray*}
\end{coro}

\begin{proof}
We apply $u=\frac{1}{a} x^a$ in Theorem~\ref{theoplapx} and we obtain inequality~\eqref{hardypx} with measures $\widetilde{\mu}_1,\,\widetilde{\mu}_2$. We simplify the right--hand side measure $\widetilde{\mu}_2$ by taking $\sigma(x)=\beta-\frac{2}{a}(p(x)-1)$ which satisfies crucial conditions.

We ensure the condition \eqref{sx} by \eqref{A}. Indeed, we estimate the expression in the left--hand side measure $\widetilde{\mu}_1$ from below as follows
\[a\beta -2(p(x)-1)+  (1-a)\left[  x \log x  p'(x)-(p(x)-1)\right]=\]
\[=a\beta+(1-a) x \log x p'(x)+(a-3)(p(x)-1)\geqslant A(x)\geqslant 0.\]

We reach the goal by dividing both sides by $a^\beta$.
\end{proof} 

\begin{rem}[Comparison of Theorem~\ref{theoHaHaKo1d} and Corollary~\ref{coropx1ha2}]\label{remHasto}  \rm
Inequalities~\eqref{hhk2} and our~\eqref{hardypx1ha} are similar, however there are some differences. Inequality \eqref{hardypx1ha} is a modular version, while \eqref{hhk2} is a norm one and it involves  the additional term as well as the weights $\mu_{1,\beta}$ and $\mu_{2,\beta}$ of power type with strictly negative exponents. The requirements on $p(x)$ are of different types.
Furthermore, we formulate  inequality \eqref{hardypx1ha} for every Lipschitz and compactly supported function $\xi$ in $I$, while~\eqref{hhk2} is stated for  $\xi \in W^{1,p(x)} (I)$ with $\xi(0)=0$.
Moreover, we allow infinite interval $I$ and a bit different range of parameter~$a$.
\end{rem}

\begin{rem}\rm The following functions $p(x)$ are admissible both in Theorem~\ref{theoHaHaKo1d} and in Corollary~\ref{coropx1ha2}.
We note that in Theorem~\ref{theoHaHaKo1d} we need to restrict our consideration to the interval $I=(0,M)$, $M<\infty$. To compare with Corollary~\ref{coropx1ha2}, in the two last examples we
allow the infinite interval~$I$.
\begin{itemize}
\item If $\gamma > 1$
,
we take $p(x)=x+\gamma$.
\item If  $\gamma \geqslant 1$
,
we take $p(x)=2-\frac{1}{x+\gamma}$.
\item If $\gamma>0$ and $d_1>d_2>0$
,
we take $p(x)=1+\frac{\gamma+d_1 x}{\gamma+d_2 x}$.

\end{itemize}
In every example of the above ones, in our~\eqref{hardypx1ha} in $\mu_{1,\beta}$ we may choose $A(x)$ separated from zero.
\end{rem}

The result of \cite{HaHaKo} was further developed in variable exponent Orlicz--Sobolev setting~\cite{miz2}.

\subsubsection*{ Results of Mashiyev--\c{C}eki\c{c}--Mamedov--Ogras \cite{MCMO}}
In \cite{MCMO} the authors prove the following extension of Hardy--type inequality from \cite{HaHaKo} by Harjulehto--H\"{a}st\"{o}--Koskenoja described above.

\begin{theo}[{\cite[Theorem~3]{MCMO}}]\label{theomcmo} Suppose $p(x),$ $q(x)$ and $\alpha(x)$  are log--H\"older continuous at the origin and at the infinity, i.e. there exist constants $C_i,$ $i=1,2,$ such that the following conditions hold
\[
|p(x)-p(0)| \log \frac{1}{x} \leqslant C_1, \qquad {\rm where }\;  x\in \left(0, {1}/{2}\right]
\]
and
\[
|p(x)-\lim_{|x|\to\infty}p(x)|\log(e+x) \leqslant C_2, \qquad {\rm where } \; x\in(0,\infty),
\]
with $1<p^- \leqslant p(x) \leqslant q(x) \leqslant q^+<\infty$ and $-\infty <\alpha^-\leqslant \alpha(x)<\infty$ for $x \in (0,\infty)$.
Then there exists a constant $C>0$ such that for every  function~$\xi$, absolutely continuous on $[0,\infty)$, with $\xi(0)=0$ we have
\begin{equation}\label{mcmo}
\|\xi(x)x^{\alpha(x)-\frac{1}{p'(x)}-\frac{1}{q(x)}}\|_{L^{q(x)}(0,\infty)} \leqslant C\|\xi'(x)x^{\alpha(x)}\|_{L^{p(x)}(0,\infty)}.
\end{equation} \end{theo}

 In \cite{HaHaKo} the authors prove \eqref{mcmo} with constant $\alpha$, $q(x)=p(x)>1$, on a finite interval $I$, and without the assumption $p(0) \leqslant p(x)$ for small $x$'s.

\medskip

We have the following related result.

\begin{rem}\rm When in Corollary \ref{coropx1ha2} we assume additionally that $A(x)$ has values separated from zero, we take $I=\R_+$, we put $\alpha(x)=a\left(1-\frac{\beta+1}{p(x)}\right)$, and we rearrange power--type terms, we obtain
\begin{eqnarray*}
&A_0\left(\int_{0}^\infty |x^{\alpha(x)-1}\xi|^{p(x)}\right) dx \leqslant \int_{0}^\infty |x^{\alpha(x)} \xi'|^{p(x)} + \left( x^{\alpha(x)}  |\xi{\log \xi|}  \frac{|p'(x)| }{p(x)}\right)^{p(x)}  dx&
\end{eqnarray*}
for every Lipschitz function $\xi$ with compact support in $\mathbb{R_+}$.

The comparison of Theorem~\ref{theomcmo} with our above inequality is similar as in the case of  theorem by Harjulehto--H\"{a}st\"{o}--Koskenoja \cite{HaHaKo} (see Remark~\ref{remHasto}).
\end{rem}

\subsubsection*{ Results of Diening--Samko~\cite{DiSa}, Rafeiro--Samko~\cite{RaSa}, Harman~\cite{Ha} and others}
In \cite{DiSa} the derived Hardy--type inequality involves Hardy operator. For $p\in \mathcal{P}(0,\infty)$ satisfying  conditions related to log--H\"older continuity, the authors prove the following inequality
\[
\Big\| x^{\alpha(x)+\mu(x)-1} \int_0^x\frac{\xi(y)}{y^{\alpha(y)}} dy\Big\|_{L^{q(x)}(0,\infty)} \leqslant C \|\xi\|_{  L^{p(x)}(0,\infty)},
\]
where $x \in (0,\infty)$, an exponent $q \in \mathcal{P} (0,\infty)$ is any function such that $\frac{1}{q(0)}=\frac{1}{p(0)}-\mu(0)$ with $\mu(0) \in [0, \frac{1}{p(0)})$, $\frac{1}{q(\infty)}=\frac{1}{p(\infty)}-\mu(\infty)$ with $\mu(\infty) \in [0, \frac{1}{p(\infty)})$ and $\alpha(0) <\frac{1}{p'(0)}, \alpha(\infty) <\frac{1}{p'(\infty)}$. Exponents $q(x), \mu(x)$ and $\alpha(x)$ are also supposed to satisfy local log--H\"{o}lder condition in zero and  infinity.

In \cite{RaSa} by Rafeiro--Samko the derived Hardy--type inequality involves the Riesz potential. It is stated on a bounded domain $\Omega \subset \mathbb{R}^n$, which complement has the cone property. Similar Hardy--type inequality is considered in~\cite{MaHa}.

An inequality corresponding to results of \cite{RaSa}, but involving Hardy operator $Hv(x)=\int_0^x v(t) dt$, is proven in~\cite{Ha}. The authors derive the following inequality which holds for every nonnegative and locally integrable function~$\xi$
\begin{equation} \label{Ha}
\| H\xi|x|^{\alpha(x)-1}\|_{L^{p(x)}(0,l)} \leqslant \|\xi|x|^{\alpha(x)}\|_{L^{p(x)}(0,l)},
\end{equation}
where $l>0$, functions $\alpha,p : (0,l) \rightarrow \mathbb{R}$ are  measurable and such that $-\infty <\alpha^-\leqslant \alpha(x) \leqslant \alpha^+<\infty$ and $-\infty <p^-\leqslant p(x) \leqslant p^+<\infty$. Moreover, the author indicate the necessary condition for validity of Hardy inequality \eqref{Ha} (see \cite[Theorem 3, 4]{Ha}).

There are several other papers dealing with one--dimensional Hardy inequality involving Hardy operator,  e.g.~\cite{cruzuribe,harman2,HaMa,harman3,MaZe1,MaZe2}. Those papers consider the further regularity analysis of inequality similar to~\eqref{Ha}, with different kind of weights under norm. The authors indicate the different type of regularity in the neighborhood of zero and at infinity for the variable exponents and present the necessary and sufficient conditions for the validity of Hardy inequality.

\section{Open questions}\label{SecOpen}

We find it interesting to investigate the following ideas.

\subsubsection*{Erasing the additional term}
Is it possible to improve~\eqref{hardypx} to an inequality of the following form
\[
\int_\Om \left|\xi \right|^{p(x)} \mu_{1,\beta}(dx)\leqslant c_2\int_\Om |\nabla \xi|^{p(x)}\mu(dx),
\]
where $c_2>0$ and $\mu_{1,\beta}(dx)$ is given by \eqref{mu1px}, and $\mu(dx)$ is eventually worse that $\mu_{2,\beta}(dx)$ given by \eqref{mu2px}? 

\subsubsection*{Improving the right--hand exponent}
We find it deserving attention to improve an exponent on the right--hand side of \eqref{hardypx}. When is it possible to prove an inequality \begin{equation*}
\int_\Om \ |\xi|^{q(x)} \mu_{1,\beta}(dx)\leqslant \int_\Om |\nabla \xi|^{p(x)}\mu_{2,\beta}(dx)+\int_\Om \left|\xi{\log \xi} \right|^{p(x)}  \mu_{3,\beta}(dx),
\end{equation*}
with $q(x)>p(x)$?

\subsubsection*{Inequalities in more general spaces}
Investigation on Hardy inequalities in variable exponent spaces is lively studied topic. They are considered for Orlicz--Sobolev functions with $|\nabla u|\in L^{p(\cdot)}\log L^{p(\cdot)q(\cdot)}$  in the literature~\cite{miz2}. Our framework is already applied in Orlicz setting~\cite{orliczhardy} with constant type of growth. What inequalities can be obtained in variable exponent Orlicz--Sobolev via our method? 



\subsection*{Acknowledgments}

The authors would like to thank Agnieszka Ka\l{}amajska, Tomasz Adamowicz, and Lech Maligranda for  discussions and help in finding appropriate literature.


\begin{thebibliography}{99}

\bibitem{AdamHa} T. Adamowicz, P. H\"{a}st\"{o}, {\em Harnack's inequality and the strong $p(\cdot)$--Laplacian}, J. Differ. Equations 250 (3) (2011), 1631--1649.

\bibitem{barbhardy04} G. Barbatis, S. Filippas, A. Tertikas, {\em A unified approach to improved $L^p$ Hardy inequalities with best constants,} Trans. Amer. Math. Soc. 356 (6) (2004), 2169--2196.

\bibitem{barnas} S. Barna\'s, {\it Existence result for hemivariational inequality involving $p(x)$--Laplacian}, Opuscula Math. 32 (2012), 439--454.


\bibitem{bogdan} K. Bogdan, B. Dyda, {\it The best constant in a frictional Hardy inequality}, Math. Nachr. 284 (5) (2011), 629--638.

\bibitem{hajpunkt}  B. Bojarski, P. Haj\l{}asz, {\em Pointwise inequalities for Sobolev functions and some applications}, Studia Math.  106 (1993), 77--92.

\bibitem{Bo} S. Boza, J. Soria, {\em Weighted Hardy modular inequalities in variable $L^p$ spaces for decreasing functions}, J. Math. Anal. Appl. 348 (2008), 383--388.

\bibitem{bhs} S. M. Buckley, R. Hurri--Syrj\"{a}nen, {\em Iterated log--scale Orlicz--Hardy inequalities}, Ann. Acad. Sci. Fenn. Math. 38 (2) (2013), 757--770.

\bibitem{buc} S. M. Buckley, P. Koskela, {\em Orlicz--Hardy inequalities,}  Illinois J. Math. 48 (3) (2004), 787--802.

\bibitem{cac} R. Caccioppoli, {\em Limitazioni integrali per le soluzioni di un'equazione lineare ellitica a derivate parziali,} Giorn. Mat. Battaglini 80 (1951), 186--212.

\bibitem{cruz} D. V. Cruz--Uribe, A. Fiorenza, {\em Variable Lebesgue Spaces. Foundations and Harmonic
Analysis}, Birkh$\ddot{\rm{a}}$user, Springer, Heidelberg, 2013. 

\bibitem{cruzuribe} D. V. Cruz--Uribe, F. I. Mamedov, {\em On a general weighted Hardy type inequality in
the variable exponent Lebesgue       spaces}, Rev. Mat. Complut. 25 (2) (2012), 335--367.

\bibitem{dambr}L. D'Ambrosio, {\em Hardy type inequalities related to degenerate elliptic differential operators}, Ann. Scuola Norm. Sup. Pisa Cl. Sci. ser. 5,  IV (2005), 451--486.

\bibitem{akiraj}  R. N. Dhara, A. Ka\l{}amajska,
{\em On equivalent conditions for the validity of Poincar\'e
inequality on weighted Sobolev space with applications to
the solvability of degenerated PDEs involving p-Laplacian,} preprint, http://www.mimuw.edu.pl/badania/preprinty/preprinty-imat/

\bibitem{ks} L. Diening, P. Harjulehto, P. H\"{a}st\"{o}, M. R{u}\v{z}i\v{c}ka, {\em Lebesgue and Sobolev Spaces with Variable Exponents}, Lecture Notes in Math. 2017, Springer--Verlag, Heidelberg, 2011.

\bibitem{DiSa} L. Diening, S. Samko, {\em Hardy inequality in variable exponent Lebesgue spaces}, Fract. Calc. Appl. Anal. 10 (1) (2007),  1--18.

\bibitem{barskrzy2} S. Dudek, I. Skrzypczak, {\em Variable exponent Hardy inequalities in $\rn$,} preprint.

\bibitem{fanzhang} X. Fan, Q. Zhang, {\it Existence of solutions for $p(x)$--Laplacian Dirichlet problem}, Nonlinear Anal. 52 (2003), 1843--1852.

\bibitem{fan} X. Fan, D. Zhao, \textit{On the generalized Orlicz--Sobolev space $W^{k,p(x)}(\Omega)$}, J. Gansu Educ. College 12 (1) (1998), 1--6.

\bibitem{orlicz} X. Fan, D. Zhao, \textit{On the spaces $L^{p(x)}(\Omega)$ and $W^{m,p(x)}(\Omega)$}, J. Math. Anal. Appl. 263 (2001), 424--446.

\bibitem{HaHaKo} P. Harjulehto, P. H\"{a}st\"{o}, M. Koskenoja, {\em Hardy's inequality in a variable exponent Sobolev spaces}, Georgian Math. J. 12 (3) (2005), 431--442.


\bibitem{overview} P. Harjulehto, P. H\"{a}st\"{o}, U. Le, M. Nuortio, {\em Overview of differential equations with non--standard growth}, Nonlinear Anal. 72 (2010), 4551--4574.

\bibitem{Ha} A. Harman, {\em On necessary condition for the variable
exponent Hardy inequality}, J. Func. Sp. Appl. (2012),
http://dx.doi.org/10.1155/2012/385925.

\bibitem{harman2} A. Harman, {\em On necessary and sufficient conditions for variable exponent Hardy
inequality}, Math. Inequal.
      Appl. 17 (1) (2014), 113--119.
      
\bibitem{HaMa} A. Harman, F. I. Mamedov, {\em On boundedness of weighted Hardy operator in
$L^p(\cdot)$ and regularity condition}, J. Inequal. Appl. (2010),
Art. ID 837951, 14 pp.

\bibitem{harman3} A. Harman, F. I. Mamedov, {\em On a Hardy type general weighted inequality in spaces
$L^{p(\cdot)}$}, Integral       Equations Operator Theory 66 (4) (2010), 565--592.


\bibitem{Hudzik} H. Hudzik, {\em On generalized Orlicz--Sobolev
space}, Funct. Approximatio Comment. Math. 4 (1976), 37--51. 

\bibitem{iwsbo} T. Iwaniec, C. Sbordone, {\em Caccioppoli estimates and very weak solutions of elliptic equations. Renato Caccioppoli and modern analysis.}  Atti Accad. Naz. Lincei Cl. Sci. Fis. Mat. Natur. Rend. Lincei (9) Mat. Appl. 14 (2003), 3 (2004), 189--205.

\bibitem{akkpp2012} A. Ka\l amajska, K. Pietruska--Pa\l uba, {\em New Orlicz variants of Hardy type inequalities with power, power--logarithmic, and power--exponential weights,} Cent. Eur. J. Math. 10 (6) (2012), 2033--2050.

\bibitem{nonex} A. Ka\l amajska, K. Pietruska--Pa\l uba, I. Skrzypczak, {\em Nonexistence results for differential inequalities involving $A$--Laplacian,} Adv. Diff. Eqs. 17 (3--4) (2012), 307--336.

\bibitem{KoRa} O. Kov\'{a}cik and J. R\'{a}kosn\'{i}k, {\em On spaces $L^{p(x)}$ and $W^{1,p(x)}$}, Czechoslovak Math. J. 41 (116) (1991), 592--618.

\bibitem{kmp} A. Kufner, L. Maligranda, L. E. Persson, {\em  The Hardy inequality. About its history and some related results}, Vydavatelsk\'{y} Servis, Plze\v{n}, 2007.



\bibitem{miz2} F.-Y. Maeda, Y. Mizuta, T. Ohno, T. Shimomura, {\em Hardy's inequality in Musielak--Orlicz--Sobolev spaces}, Hokkaido Math. J. 44 (2) (2014), 139--155.

\bibitem{MaHa} F. I. Mamedov, A. Harman, {\em
On a weighted inequality of Hardy type in spaces $L^{p(\cdot)}$},  J. Math. Anal. Appl. 353 (2) (2009), 521--530.

\bibitem{MaZe1} F. I. Mamedov, Y. Zeren, {\em On equivalent conditions for the general weighted Hardy type inequality in  space      $L^{p(\cdot)}$},  Z. Anal. Anwend. 31 (1) (2012), 55--74.

\bibitem{MaZe2}  F. I. Mamedov, Y. Zeren, {\em A necessary and sufficient condition for Hardy's operator in the variable Lebesgue   space}, Abstr. Appl. Anal. 2014, Art. ID 342910, 7 pp.

\bibitem{MCO} R. Mashiyev, B. \c{C}eki\c{c},  S. Ogras,
{\em On Hardy's inequality in $L^{p(x)}(0,\infty)$}
JIPAM. J. Inequal. Pure Appl. Math. 7 (3) (2006), 1--5.

\bibitem{MCMO} R. Mashiyev, B. \c{C}eki\c{c}, F. I. Mamedov, S. Ogras,
{\em Hardy's inequality in power--type weighted $L^{p(x)}(0,\infty)$},
 J. Math. Anal. Appl. 334 (1) (2007), 289--298.

\bibitem{pohmi_99} E. Mitidieri, S. Pohozaev,  {\em  Nonexistence of positive solutions to quasilinear elliptic problems in $\rn$,} Proc. Steklov. Inst. Math. { 227} (1999), 186--216, (translated from Tr. Mat. Inst.
Steklova { 227} (1999), 192--222).

\bibitem{miz1} Y. Mizuta, E. Nakai, T. Ohno, T. Shimomura, {\em Hardy's inequality in Orlicz--Sobolev spaces of variable exponent}, Hokkaido Math. J. 40 (2) (2011), 187--203.


\bibitem{muckenhoupt} B. Muckenhoupt, {\em Hardy's inequality with weights}, Studia Math. 44 (1972), 31--38.

\bibitem{nakano1} H. Nakano, {\em  Modulared Semi--ordered Linear Spaces}, Maruzen Co., Ltd., Tokyo, 1950.

\bibitem{nakano2} H. Nakano, {\em Topology and Linear Topological Spaces}, Maruzen Co., Ltd., Tokyo, 1951.

\bibitem{Orlicz} W. Orlicz, {\em \"{U}ber konjugierte Exponentenfolgen}, Studia Math. 3 (1931), 200--211.

\bibitem{RaSa} H. Rafeiro, S. Samko, {\em Hardy type inequality in variable Lebesgue spaces}, Ann. Acad. Sci. Fenn. Math. 34 (2009), 279--289.

\bibitem{raj-ru1} K. Rajagopal, M. R{u}\v{z}i\v{c}ka, {\em On the modeling of electrorheological materials}, Mech. Res. Commun. 23 (1996), 401--407.

\bibitem{raj-ru2} K. Rajagopal, M. R{u}\v{z}i\v{c}ka, {\em Mathematical modeling of electrorheological materials}, Contin. Mech.
Thermodyn. 13 (2001), 59--78.

\bibitem{el-rh2} M. R{u}\v{z}i\v{c}ka, {\em Electrorheological Fluids: Modeling and Mathematical Theory}, Lecture Notes in Math. 1748, Springer--Verlag, Berlin, 2000.

\bibitem{Samko1} S. Samko, {\em Hardy inequality in the generalized Lebesgue spaces}, Fract. Calc. Appl. Anal. 6 (4) (2003), 355--362.

\bibitem{plap} I. Skrzypczak, {\em Hardy--type inequalities derived from $p$--harmonic problems}, Nonlinear Anal. TMA 93 (2013), 30--50.

\bibitem{bcp-plap} I. Skrzypczak, {\em Hardy--Poincar\'{e} type inequalities derived from $p$--harmonic problems,} Banach Center Publ. 101 Calculus of Variations and PDEs (2014), 223--236.

\bibitem{orliczhardy} I. Skrzypczak, {\em Hardy inequalities resulted from nonlinear problems dealing with $A$--Laplacian,} NoDEA Nonlinear Differential Equations Appl.  21  (6) (2014), 841--868.


\end{thebibliography}
\end{document}